\documentclass[a4paper]{article}

\title{The Module Isomorphism Problem for Finite Rings and Related Results}
\author{Iuliana Cioc\u anea-Teodorescu}
\date{}

\usepackage{amsmath}
\usepackage{amsthm}
\usepackage{amsfonts}
\usepackage{amssymb}
\usepackage{longtable}
\usepackage{pbox}
\usepackage{pdflscape}
\usepackage{longtable}
\usepackage{caption}
\usepackage{mathtools}
\usepackage{graphicx}
\usepackage{tikz}
\usepackage{tikz-cd}
\usetikzlibrary{matrix}
\usepackage{ stmaryrd }
\usepackage{enumitem}
\usepackage{xfrac}

\makeatletter
\providecommand*{\twoheadrightarrowfill@}{%
  \arrowfill@\relbar\relbar\twoheadrightarrow
}
\providecommand*{\twoheadleftarrowfill@}{%
  \arrowfill@\twoheadleftarrow\relbar\relbar
}
\providecommand*{\xtwoheadrightarrow}[2][]{%
  \ext@arrow 0579\twoheadrightarrowfill@{#1}{#2}%
}
\providecommand*{\xtwoheadleftarrow}[2][]{%
  \ext@arrow 5097\twoheadleftarrowfill@{#1}{#2}%
}
\makeatother

\theoremstyle{plain}
\newtheorem{thm}{Theorem}

\numberwithin{thm}{section}
\newtheorem{cor}[thm]{Corollary}
\newtheorem{prop}[thm]{Proposition}
\newtheorem{definition}[thm]{Definition}

\theoremstyle{definition}
\newtheorem{note}[thm]{Note}

\newtheorem*{acknowledgements}{Acknowledgements}

\DeclareMathOperator{\End}{End}
\DeclareMathOperator{\Hom}{Hom}
\DeclareMathOperator{\length}{length}
\DeclareMathOperator{\new}{new}
\DeclareMathOperator{\old}{old}
\DeclareMathOperator{\ann}{ann}

\DeclareMathOperator{\im}{im}

\DeclareMathOperator{\J}{J}

\DeclareMathOperator{\spans}{span}

\begin{document}
\maketitle

\begin{abstract}
Let $R$ be a finite ring and let $M, N$ be two finite left $R$-modules. We present two distinct deterministic algorithms that decide in polynomial time whether or not $M$ and $N$ are isomorphic, and if they are, exhibit an isomorphism. As by-products, we are able to determine the largest isomorphic common  direct summand  between two modules and the minimum number of generators of a module. By not requiring $R$ to contain a field, avoiding computation of the Jacobson radical and not distinguishing between large and small characteristic, both algorithms constitute improvements to known results. We have not attempted to implement either of the two algorithms, but we have no reason to believe that they would not perform well in practice.
\end{abstract}

\let\thefootnote\relax\footnotetext{Mathematical Institute, Leiden University, Netherlands}
\let\thefootnote\relax\footnotetext{E-mail: ciocaneai@math.leidenuniv.nl}
\let\thefootnote\relax\footnotetext{Date: October 2014}

\section{Introduction}
\label{intro}
The \emph{module isomorphism problem} (MIP) can be formulated as follows: design a deterministic algorithm that, given a ring $R$ and two left $R$-modules $M$ and $N$, decides in polynomial time whether they are isomorphic, and if yes, exhibits an isomorphism.  

This problem is as fundamental as it is easily stated and has been studied extensively, mainly due to its broad range of applications. In particular, polynomial time algorithms were given in \cite{BL,CIK} for the case where $R$ is a finite dimensional algebra over a field and $M,N$ are finite dimensional modules over that field. For our purposes, $R$ will be a finite ring (not necessarily containing a field) and $M,N$ will be finite $R$-modules. We give an algorithm as described by the following theorem:

\begin{thm}\label{mth1}
There exists a deterministic polynomial time algorithm that, given a finite ring $R$ and two finite $R$-modules $M$ and $N$, outputs maximal length direct summands of $M$ and $N$ that are isomorphic. Moreover, the algorithm produces a map $f\in\Hom_R(M,N)$ inducing such an isomorphism. 
\end{thm}

We establish the result in Theorem \ref{mth1} by a direct generalisation of the methods given in \cite{BL}, where the rings considered were finitely generated algebras over a field and the modules were finite dimensional over that field. This approach relies on the ability of finding non-nilpotent elements in non-nilpotent ideals of the endomorphism ring $\End_{R}(M)$ (cf. Proposition \ref{find}). As a direct consequence of Theorem \ref{mth1} we have that:

\begin{cor}
There exists a deterministic polynomial time algorithm that, given a finite ring $R$ and two finite $R$-modules $M$ and $N$, decides whether $M$ and $N$ are isomorphic, and if they are, exhibits an isomorphism.
\end{cor}

In Section 3, we give a second algorithm that solves the module isomorphism problem over finite rings. As a by-product, we get:

\begin{thm}\label{sideexit}
There exists a deterministic polynomial time algorithm that, given a finite ring $R$ and a finite $R$-module $M$, determines the minimum number of $R$-generators of $M$ and outputs a list of such generators. 
\end{thm}

In particular, we can determine if a module is cyclic or free (by comparing cardinalities).\\

It is important to note that the algorithms given here work for any finite ring, do not distinguish between large and small characteristic and avoid computation of the Jacobson radical, thus constituting an improvement to known results. 
Moreover, the algorithm in Theorem \ref{sideexit} is an interesting object \emph{per se}, due to its structure and the techniques it employs.
A common approach to this type of problems is to reduce to the semisimple case and then ``lift" (e.g. \cite{CIK,IKS}). In our algorithm, we work \emph{as if} the ring were semisimple and we have a list, $S_1,\ldots ,S_t$, of candidates for the isomorphism classes of simple modules composing it. During the running of the algorithm, we allow ourselves to be contradicted in our assumption about the simplicity of the $S_i$, in which case we update our list, quotient the ring by an appropriate two-sided nilpotent ideal and start again. If we are not contradicted, we may still draw conclusions. In this way, there is always a side-exit available and what forces an output in polynomial time is that we cannot take the side-exit too many times.\\

Rings are always assumed to contain a unit element, but are not necessarily commutative. Modules are always left unital, unless otherwise specified. We aim to develop the theory in as great a generality as possible, before restricting to finite rings for algorithmic purposes.

In the absence of an underlying field, a finite ring $R$ is assumed to be given to the algorithm via its group structure. Let $R^{+}$ be the underlying additive group of $R$. Then 
\begin{equation}
R^{+}\cong \bigoplus_{i=1}^t \mathbb{Z}/n_i\mathbb{Z},
\end{equation}
for some $t\in\mathbb{Z}_{\geq 0}$, $n_1,\ldots ,n_t \in \mathbb{Z}_{> 0}$. So, in order to specify $R^{+}$, we give the algorithm a sequence of integers $n_1,\ldots ,n_t$ . To make it into a ring, we also give a bilinear map 
\begin{equation}
R^{+}\times R^{+}  \to R^{+},\quad (e_i,e_j)  \mapsto e_ie_j,
\end{equation}
where $\lbrace e_i\rbrace_{i=1}^t$ are the additive generators of $R^{+}$ and we express $e_ie_j$ linearly in terms of the $e_k$'s. This amounts to giving $t^3$ numbers.

To input a module, we give a finite abelian group $(M,+)$, as before, and a bilinear map
$R\times M\to M$,
where, for every additive generator of $R$ and $M$, we express the image in terms of the additive generators of $M$.

\section{MIP via non-nilpotent endomorphisms}
\label{sec:1}

\subsection{Finding non-nilpotent elements}
\label{sec:2}

An ideal $I$ of a ring $R$ is said to be \emph{nil} if all its elements are nilpotent. The following is a well known fact:

\begin{prop}
If $R$ is a left-Artinian ring, then a left (or right) ideal $I$ of $R$ is nil if and only if it is nilpotent.
\end{prop} 
\begin{proof}
Clearly if a left (or right) ideal is nilpotent, then it is nil. For the other direction, note that any nil left (or right) ideal of $R$ is contained in the Jacobson radical $\J(R)$ (see \cite{lam2001first}, Lemma 4.11). So it is enough to show that $J:=\J(R)$ is nilpotent.

Suppose $J$ is not nilpotent. Since $R$ is left-Artinian, the sequence of consecutive powers of $J$ must stabilize, i.e. $\exists n\in \mathbb{Z}_{> 0} \text{ such that } J^n=J^{2n}\neq 0$. Then there exists a left ideal $L$ s.t. $J^nL\neq 0$. Suppose $L$ is minimal with this property. Since $J^nL\neq 0$, there exists $a\in L$ such that $J^na\neq 0$. Now $Ra\subseteq L$ and $J^n a\subseteq L$. Moreover, $0\neq J^na=J^{n}J^{n}a$ and $0\neq J^n a\subseteq J^n Ra$, so by minimality of $L$, we have that $L=Ra=J^na$. Hence $a=xa$, for some $x\in J^n$. It follows that $a=0$ (otherwise, $1-x$ would be both a unit and a left zero-divisor, which is not possible). This gives the desired contradiction.
\end{proof}

If $R$ is finite, the proof above can be translated into an algorithm for finding a non-nilpotent element in a non-nilpotent left ideal of $R$:

\begin{prop}\label{find}
There exists a deterministic polynomial time algorithm that, given a finite ring $R$ and a left ideal $I$, determines whether or not $I$ is nilpotent and if it is not, produces a non-nilpotent element lying inside it.
\end{prop}
\begin{proof}
Suppose $R$ is a $k$-algebra, where $k$ is a commutative ring ($R$ will certainly be an algebra over its centre or its prime subring) and let $I$ be a non-nilpotent left ideal of $R$.  Since $R$ is finite, we can find $n\in\mathbb{Z}_{>0}$ such that $I^n=I^{2n}\neq 0$. Suppose $I^n$ is generated over $k$ by a set $A$ and over $R$ by a set $B$, i.e. 
\begin{equation}
I^n=\sum_{\alpha\in A} k\alpha=\sum_{\beta\in B} R\beta.
\end{equation}
Then $0\neq I^{2n}=\sum_{\alpha\in A \atop \beta\in B} R\beta\alpha$, so there exists $b\in B$ and $a\in A$ such that $b\cdot a\neq 0$.
Consider $I^n a\subseteq Ra$ (note $I^na\neq 0$ since it contains $ba$). If equality holds, then we can write $a=xa$, for some $x\in I^n$ and since $a\neq 0$, it must be the case that $x$ is non-nilpotent, as required (otherwise $1-x$ would be a unit and a left zero divisor at the same time, which is impossible). Suppose now that the inclusion is strict. Then there exists $c\in I^na$ such that $I^n c\neq 0$ (otherwise $I^nI^na=I^na=0$, which is a contradiction, since $0\neq ba\in I^{n}a$). 
Moreover, we will be able to find such a $c$ among the $k$-generators of $I^n a$. 
We have now produced a smaller ideal $Ra\supsetneq I^na\supset Rc$ such that $I^nc\neq 0$. We replace $a$ by $c$ and keep going. This process must terminate after a finite number of steps, since $R$ is finite.
We also note that $\length({_R}R)\leq \log_2{\vert R \vert}$ and $I$ is nilpotent if and only if $I^{\length({_R}R)}=0$, so it is possible to test if $I$  is nilpotent. 
\end{proof}

\subsection{Splitters}
Let $R$ be an algebra over a commutative Artinian ring $k$ such that $R$ is finitely generated as a $k$-module. Let $M_1$ and $M_2$ be two $R$-modules of finite length over $k$. Note that $\Hom_R(M_1,M_2)$ is a left $k$-module and a right $\End_{R}(M_1)$-module. Moreover, $\End_{R}(M_1)$ is left (and right) Artinian. 
\\  

Following \cite{BL}, we make the following definition:
\begin{definition}
Let $f\in\Hom_{R}(M_1,M_2)$. 
A decomposition $M_1=N_1\oplus K_1$, for $N_1,K_1\leq M_1$, is called an \emph{$f$-decomposition} if $N_1\neq 0$, $\ker(f)\leq K_1$ and the image of $N_1$ under $f$, which we denote by $fN_1$, is a direct summand of $M_2$.
If $M_1$ has an $f$-decomposition, we say that $f$ is a \emph{splitter}.
\end{definition}

The following proposition and its proof, together with Proposition \ref{find}, allow us to algorithmically decide if a given homomorphism $f$ is a splitter, and if it is, to produce an $f$-decomposition:

\begin{prop}\label{splitternilp}
Let $f\in\Hom_{R}(M_1,M_2)$. Then $f$ is a splitter if and only if there exists $g\in \Hom_{R}(M_2,M_1)$ such that $gf$ is not nilpotent.
\end{prop}
\begin{proof}
The proof of this proposition is the same as the one given in \cite{BL}, Lemma 3.3, which treats the case when $R$ is a finitely generated algebra over a field and $M_1,M_2$ are finite dimensional over that field. For completeness we include the proof of the ``if" statement here.

  
Suppose $g\in\Hom_R(M_2,M_1)$ is such that $gf$ is not nilpotent. Let $s=gf$ and $t=fg$. Since $M_1$ and $M_2$ have finite length over $k$ (so are both Artinian and Noetherian over R), we can apply Fitting's Lemma (see \cite{lang}, Chapter X, Proposition 7.3) to say that $M_1=\ker(s^d)\oplus \im(s^d)$, $M_2=\ker(t^d)\oplus \im(t^d)$, for $d=\max\lbrace \length_R(M_1),\length_R(M_2)\rbrace$ and the restriction of $t$ to $\im(t^d)$ is an automorphism. We have that $s^dM_1\neq 0$ since $s$ is not nilpotent, $\ker(f)\leq \ker(s^d)$ by definition of $s$ and
\begin{equation*}
f(s^dM_1)=t^d(fM_1)\subseteq t^dM_2 = fg(t^dM_2)=fs^d(gM_2 )\subseteq f(s^dM_1), 
\end{equation*}
so $f\im(s^d)=f(s^dM_1) = t^dM_2 =\im(t^d)$, which is what is required for $f$ to be a splitter. 
\end{proof}

Suppose now that we are given $f\in \Hom_R(M_1,M_2)$ and we wish to know whether it is a splitter. To do this, we first compute a set $C$ of $k$-generators of $\Hom_{R}(M_2,M_1)$. Now consider the left ideal $I$ of $\End_R(M_1)$ generated by the set $Cf=\lbrace cf \mid c\in C\rbrace$. If $I$ is nilpotent (which we can determine by Proposition \ref{find}), then since $\Hom_{R}(M_2,M_1)f=\spans _k (Cf)= I$, we will not be able to find a non-nilpotent element of the form $gf$, so $f$ cannot be a splitter. Otherwise the algorithm of Proposition \ref{find} will produce some $g\in\Hom_R(M_2,M_1)$ witnessing that $f$ is a splitter and we can produce an $f$-decomposition by Proposition \ref{splitternilp}.\\

We now have a way of identifying splitters. But we would not like to have to look for them over all homomorphisms. The following proposition tells us that we can restrict our attention to a considerably smaller set: 
\begin{prop}\label{lookfor}
If a splitter exists, then there exists a splitter in any set of $k$-module generators of $\Hom_{R}(M_1,M_2)$ and in any set of $\End_{R}(M_1)$-generators of $\Hom_{R}(M_1,M_2)$.
\end{prop}
\begin{proof}
To see that a set of $\End_R(M_1)$-module generators is enough, note that 
\begin{equation*}
f \text{ is not a splitter } \iff \Hom_R(M_2,M_1)f\subseteq \J(\End_R(M_1)). 
\end{equation*}
Let $B$ be a set of $\End_R(M_1)$-module generators of $\Hom_R(M_1,M_2)$. Suppose that $B$ does not contain any splitters. Then
\begin{align*}
\Hom_R(M_2,M_1)\Hom_R(M_1,M_2) & =\sum_{b\in B}\Hom_R(M_2,M_1)b\End_R(M_1) \\
& \subseteq \J(\End_R(M_1)),
\end{align*}
and therefore $\Hom_R(M_1,M_2)$ cannot contain a splitter.
 
Finally, note that any set of left-$k$-module generators is also a set of right-$\End_R(M_1)$-module generators.
\end{proof}


Putting these results together, we can construct an algorithm satisfying the requirements of Theorem \ref{mth1} as follows. We view $R$ as an algebra over its prime subring $k$. Let $M$ and $f$ be two auxiliary variables that at the end of the algorithm will become equal to the desired maximal length direct summand of $M_1$ that has an isomorphic copy as a direct summand of $M_2$, and the isomorphism-inducing $f\in\Hom_R(M_1,M_2)$.
At the beginning of the algorithm, we put $M$ and $f$ equal to zero. We compute $B$, a set of $k$-generators of $\Hom_{R}(M_1,M_2)$ (or a set of $\End_{R}(M_1)$-module generators thereof). For each element of $B$, we test if it is a splitter: by Proposition \ref{lookfor}, if a splitter exists, we will find one inside $B$. 
Finding a splitter $b\in B$ also gives us a decomposition $M_1=N_1\oplus K_1$ and $M_2=N_2\oplus K_2$ with $N_1\cong N_2$ via $b$. We make the following replacements: $M:=M\oplus N_1$ and $f:=f\oplus (\text{the restriction of $b$ to $N_1$})$, $M_1:=K_1$, $M_2:=K_2$, $B:=\text{set of $k$-module generators of the new $\Hom_{R}(M_1,M_2$})$ and we start again. Note that we are all the time assuming the Krull-Remak-Schmidt Theorem (\cite{lang}, Chapter X, Theorem 7.5), which ensures existence and uniqueness up to isomorphism of the direct summands of $M_1$ and $M_2$.

\section{MIP via an approximation of the Jacobson radical}

Let $R$ be a finite ring. We observe that determining whether two finite $R$-modules $M_1$ and $M_2$ are isomorphic reduces to determining if a module is free of rank one. Let $E:=\End_{R}(M_1)$ and $K:=\Hom_{R}(M_2,M_1)$. Note that $K$ is a left $E$-module.

\begin{prop}\label{tfae}
Let $R,M_1,M_2,E,K$ be as above. The following are equivalent:

(i) $M_1\cong M_2$ as $R$-modules. 

(ii) $E\cong K$ as $E$-modules and the image of $1_E$ in $K$ under any such isomorphism is an isomorphism between $M_1$ and $M_2$.

(iii) There exists an $E$-module isomorphism $\phi: E\rightarrow K$ such that $\phi(1_E)$ is an isomorphism between $M_1$ and $M_2$.
\end{prop}
\begin{proof}
The implications $(ii\Rightarrow iii)$ and $(iii\Rightarrow i)$  are immediate.
For $(i\Rightarrow ii)$, suppose $f:M_1\xrightarrow{\sim} M_2$. Then $E\cong K$, so let $\phi:E\xrightarrow{\sim}K$ be any such isomorphism. Let $\lambda:=\phi(1)$. Then there exists a unique $\epsilon\in E$ such that 
\begin{equation*}
f=\phi(\epsilon)=\phi(\epsilon \cdot 1)=\epsilon \phi(1)=\epsilon \lambda,
\end{equation*}
where the third equality follows by $E$-linearity of $\phi$. 
Since $f$ is injective, so must $\lambda$ be. Moreover, since $M_1$ and $M_2$ have the same length, $\lambda$ must also be surjective.
\end{proof}

Hence an algorithm that can establish freeness of rank one of a module provides a solution to the module isomorphism problem. 

\subsection{Computing minimum number of generators}
Let $R$ be a left-Artinian ring. Suppose, along with $R$, we are given a collection $S_1,\ldots , S_t$ of nonzero left $R$-modules of finite length. Ideally, we would like this collection to be a set of representatives for the isomorphism classes of simple $R$-modules. However, for now, we only require that each simple $R$-module occurs in at least one of the $S_i$, i.e. it occurs as a quotient in its composition series. We can take, for example, $t=1$ and $S_1={_R}R$.

Let 
\begin{equation}
\mathfrak{a}=\bigcap_{i=1}^t \ann_R(S_i),
\end{equation} 
where $\ann_R(S_i)=\ker(R\to\End_{\mathbb{Z}}(S_i))$. Again, ideally we would like $\mathfrak{a}$ to be the Jacobson radical, $\J(R)$. In reality, we only have one inclusion, namely $\mathfrak{a}\subseteq \J(R)$, since an element of $\mathfrak{a}$ will kill all the $S_i$'s and so will kill all submodules, quotients and submodules of quotients of the $S_i$. Since every simple $R$-module occurs in at least one of the $S_i$, we have that $\mathfrak{a}$ kills all simple $R$-modules, which is equivalent to being inside the Jacobson radical. Furthermore, since $R$ is left-Artinian, $\J(R)$ is nilpotent and hence $\mathfrak{a}$ is nilpotent.

We will construct an algorithm that, given $R$ and a collection of $S_i$, either ``improves" the sequence of $S_i$ or computes $\mathfrak{a}$ and an isomorphism of $R$-modules
\begin{equation}\label{iss}
R/\mathfrak{a}\xrightarrow{\sim} \bigoplus_{i=1}^t S_i^{a_i},
\end{equation}
for suitable $a_i\in\mathbb{Z}_{\geq 0}$.
Similarly, we construct an algorithm that, when also given a finitely generated $R$-module $M$, either ``improves" the sequence of $S_i$ or computes an isomorphism of $R$-modules
\begin{equation}
M/\mathfrak{a}M\xrightarrow{\sim} \bigoplus_{i=1}^t S_i^{c_i},
\end{equation}
for suitable $c_i\in\mathbb{Z}_{\geq 0}$.
Consider the quantity $l{(S_i)_{i=1}^t}=\sum_{i=1}^t(2\length(S_i)-1)\in\mathbb{Z}_{\geq 0}$. An ``improvement" in the sequence is measured by a decrease in $l{(S_i)_{i=1}^t}$ and occurs either when we remove one of the $S_i$ from the list (either because the list already contains an isomorphic copy of it or because it is not needed) or when we discover a nonzero proper submodule $T$ of one of the $S_i$ which witnesses nonsimplicity of $S_i$ and which we use to replace $S_i$ by $S_i/T$ and $T$, that now stand a better chance of being simple. Note that the factor 2 in the expression of $l{(S_i)_{i=1}^t}$ ensures it decreases even when we remove an $S_i$ from the list. \\

Let us first write out the details of the routine that finds an isomorphism as in (\ref{iss}). We will call this routine UPDATE and within the main algorithm we will call it whenever we have improved on our sequence of $S_i$'s and we want to update our $\mathfrak{a}$, the sequence of $a_i$'s and the isomorphism $R/\mathfrak{a}\xrightarrow{\sim} \bigoplus_{i=1}^t S_i^{a_i}$.

Let $\mathfrak{S}=\lbrace (S_i)_{i=1}^t \mid t\in \mathbb{Z}_{\geq 0}$, each $S_i$ is a nonzero finite length $R$-module and each simple $R$-module occurs as a factor in the composition series of at least one $S_i\rbrace$.

\begin{prop}\label{aa1}
There exists a deterministic polynomial time algorithm that takes as input a finite ring $R$ and a collection of modules $(S_i)_{i=1}^t\in\mathfrak{S}$ and outputs a sequence of integers $a_1,\ldots ,a_t\in\mathbb{Z}_{>0}$, a two-sided nilpotent ideal $\mathfrak{a}'$ of $R$ and an isomorphism $\varphi:R/\mathfrak{a}'\xrightarrow{\sim} \bigoplus_{i=1}^{t'} (S_i')^{a_i}$, where $t' \in\mathbb{Z}_{\geq 0}$, $\mathfrak{a}'=\cap_{i=1}^{t'}\ann_R(S_i')$ and $(S_i')_{i=1}^{t'}\in\mathfrak{S}$ is such that $l{(S_i')_{i=1}^{t'}}\leq l{(S_i)_{i=1}^t}$.
\end{prop}
\begin{proof}
Let $\mathfrak{b}$ be a left ideal of $R$ which we will use as an intermediate variable that at the end of the algorithm will become equal to the desired $\mathfrak{a}'$.
We start off by setting $\mathfrak{b}=R$, $t'=t$, $a_i=0$ and $S_i'=S_i$ for $1\leq i\leq t$, and $\varphi=0$. Throughout the algorithm $\varphi:R/\mathfrak{b}\xrightarrow{\sim} \bigoplus_{i=1}^{t'} (S_i')^{a_i}$ and $\cap_{i=1}^{t'}\ann_R(S_i')\subseteq \mathfrak{b}$ will be invariant. If for all $i$ we have $\mathfrak{b}S_i'=0$, then we are done. Otherwise, we choose $1\leq h\leq t'$ and $s\in S_h'$ such that $\mathfrak{b}s\neq 0$. We define 
\begin{equation}
\psi:R\to S_h'\oplus \bigoplus_{i=1}^{t'} (S_i')^{a_i}, \quad \psi(r)=(rs,\varphi(r)).
\end{equation}
Then $\ker(\psi)= \ann_R(s)\cap\mathfrak{b}\subsetneq\mathfrak{b}$, since $\mathfrak{b}$ did not annihilate $s$. Let $\overline{\psi}$ be the map induced by $\psi$ on $R/\ker(\psi)$. Then
\begin{equation*}
\overline{\psi} \text{ isomorphism } \iff \psi \text{ surjective } 
\iff S_h' \oplus\lbrace 0 \rbrace  \subseteq \im(\psi)
\iff \mathfrak{b}s=S_h'.  
\end{equation*}
This comes as a confirmation of our intuition: we are treating the $S_i'$ as if they were simple, so if $0\neq \mathfrak{b}s\leq S_h'$, then we would want $\mathfrak{b}s$ to be the whole of $S_h'$.

If $\psi$ is surjective, we make the following replacements:
$a_h:=a_{h}+1$, $\varphi:=\psi$ and $\mathfrak{b}:=\ker{\psi}$. Note that now $\mathfrak{b}$ has smaller length than before.

If $\psi$ is not surjective, we replace $S_h'$ by $S_h'/\mathfrak{b}s$, $t'$ by $t'+1$ and put $S_{t'+1}':=\mathfrak{b}s$ and $a_{t'+1}:=0$. In addition, we replace $\varphi$ by the composition $\pi\circ \varphi$, where $\pi$ is the canonical map $\pi:\bigoplus_{i=1}^{t'} [\old (S_i')^{\old a_i}]\to \bigoplus_{i=1}^{t'+1} [\new (S_i')^{\new a_i}]$ and we replace $\mathfrak{b}$ by the kernel of our new $\varphi$. Note that the new $\mathfrak{b}$ will contain the old $\mathfrak{b}$, but the improvement is now sitting in the new $S_i'$.

Consider the quantity $l{(S_i')_{i=1}^{t'}}=\sum_{i=1}^{t'}(2\length(S_i')-1)$. Then $0\leq t'\leq l{(S_i')_{i=1}^{t'}}$, since each $S_i'$ has length at least 1. Also, $l{(S_i')_{i=1}^{t'}}$ is bounded above by its initial value. At each iteration of the algorithm, we either see a decrease in the value of $l{(S_i')_{i=1}^{t'}}$, when we improve our sequence, or we see a decrease in the length of $\mathfrak{b}$, whose length is initially equal to $\length(R)$. Since $\length(R)\leq \log_2(\vert R \vert)$ (recall that $R$ was finite), the algorithm runs in polynomial time.
\end{proof}

We now turn to the main algorithm:

\begin{thm} 
There exists a deterministic polynomial time algorithm that, given a finite ring $R$ and a finite $R$-module $M$, determines the minimum number of $R$-generators of $M$ and outputs a list of such generators.
\end{thm}
\begin{proof}
We begin by running the UPDATE algorithm presented in Proposition \ref{aa1} with input $t=1$ and $S_1={_R}R$. 
Let $X=\lbrace x_1,\ldots, x_d\rbrace$ be a set of $R$-generators of $M$. Then $\overline{X}=\lbrace x_i+\mathfrak{a}M\mid 1\leq i\leq d\rbrace$ generates $M/\mathfrak{a}M$ over $R/\mathfrak{a}$. This gives a surjective
homomorphism $(R/\mathfrak{a})^d\cong \bigoplus_{i=1}^t S_i^{a_i d} \twoheadrightarrow M/\mathfrak{a}M$. Relabel the $S_i$ to get a map $\bigoplus_{i\in I} S_i\twoheadrightarrow M/\mathfrak{a}M$. We would like to find a subset $J\subseteq I$ for which this map becomes an isomorphism. In this process, the standard proof of its existence in the case where the $S_i$ are simple (see \cite{lang}, Chapter XVII, Lemma 2.1) may produce a witness of nonsimplicity of some $S_i$, in which case we refine the sequence and call the UPDATE subroutine to start over.
In the end, we will have produced an isomorphism $\bigoplus_{i=1}^t S_i^{c_i} \xrightarrow{\sim}M/\mathfrak{a}M $, for some $c_i\in\mathbb{Z}_{>0}$ (note that the ideal $\mathfrak{a}$ may now be different from the one we started with). 

Let $n=\max_i\lbrace \lceil \frac{c_i}{a_i}\rceil \rbrace$. 
If for all $i\neq h$ we have $\Hom_R(S_i,S_h)=0$, then there is a surjective map $(R/\mathfrak{a})^n\twoheadrightarrow M/\mathfrak{a}M$ and $n$ is minimal with this property. Since $\mathfrak{a}\subseteq \J(R)$, we can lift this to a map $R^n\twoheadrightarrow M$ and produce $n$ generators of $M$. 
If, however, for some $i\neq h$ we have that $\Hom_R(S_i,S_h)$ contains a nonzero element $f$, then we can once again refine our sequence (either using $\ker(f)\neq 0$ or $\im(f)\neq S_h$ or removing one of $S_i$ or $S_h$ from the list if they are isomorphic) and start over by calling the UPDATE routine on our newly improved sequence. 

To establish the running time, consider again the quantity $l{(S_i)_{i=1}^t}=\sum_{i=1}^t (2\length(S_i)-1)$. At each iteration of the algorithm, we either produce an output or we improve on our sequence, which results in a decrease in $l{(S_i)_{i=1}^t}$.
\end{proof}

In particular, we may use this algorithm to determine if a module is cyclic or free of rank 1 (by comparing cardinalities), thus  settling the module isomorphism problem yet again (see Proposition \ref{tfae}). 

\begin{note}
As we have described the algorithm, after we have improved our sequence, calling the UPDATE subroutine overwrites our so far acquired knowledge about $\mathfrak{a}$.
There is possibly a way of saving some of this information by running a modified version of the UPDATE routine, which would thus make the main algorithm slightly more efficient. However, being more precise about this here would obscure the idea of the algorithm.
\end{note}

\begin{acknowledgements}
I would like to thank my PhD supervisor, Professor Hendrik Lenstra, for his generous guidance and his unfailing patience, particularly in the production of the side-exit algorithm. 
\end{acknowledgements}



\end{document}